\documentclass[a4paper,12pt,reqno]{amsart}
\usepackage{amsmath, amssymb,amsthm}

\nonstopmode \numberwithin{equation}{section}
\newtheorem{theorem}{Theorem}[section]

\newtheorem{example}{Example}[section]
\newtheorem{corollary}{Corollary}[section]

\newtheorem{remark}{Remark}[section]

\newtheorem{note}{Note}

\begin{document}
\bibliographystyle{amsplain}

\title{{{
Integral transforms of functions to be in the Pascu class using duality techniques
}}}

\author{
Satwanti Devi
}
\address{
Department of  Mathematics  \\
Indian Institute of Technology, Roorkee-247 667,
Uttarkhand,  India
}
\email{ssatwanti@gmail.com}

\author{
A. Swaminathan
}
\address{
Department of  Mathematics  \\
Indian Institute of Technology, Roorkee-247 667,
Uttarkhand,  India
}
\email{swamifma@iitr.ernet.in, mathswami@gmail.com}

\bigskip

\begin{abstract}
Let $W_{\beta}(\alpha,\gamma)$, $\beta<1$, denote the class of all normalized analytic functions $f$ in the unit disc
${\mathbb{D}}=\{z\in {\mathbb{C}}: |z|<1\}$ such that
\begin{align*}
{\rm Re\,} \left( e^{i\phi}\left((1-\alpha+2\gamma)\frac{f}{z}+(\alpha-2\gamma)f'+\gamma zf''-\beta\right)\frac{}{}\right)>0, \quad z\in {\mathbb{D}},
\end{align*}
for some $\phi\in {\mathbb{R}}$ with $\alpha\geq 0$, $\gamma\geq 0$ and $\beta< 1$.
Let $M(\xi)$, $0\leq \xi\leq 1$, denote the Pascu class  of $\xi$-convex functions given by the analytic condition
\begin{align*}
{\rm Re\,}\frac{\xi z(zf'(z))'+(1-\xi)zf'(z)}{\xi zf'(z)+(1-\xi)f(z)}>0
\end{align*}
which unifies the class of starlike and convex functions.
The aim of this paper is to find conditions on $\lambda(t)$ so that the integral transforms of the form
\begin{align*}
V_{\lambda}(f)(z)= \int_0^1 \lambda(t) \frac{f(tz)}{t} dt.
\end{align*}
carry functions from $W_{\beta}(\alpha,\gamma)$ into
$M(\xi)$. As applications, for specific values of $\lambda(t)$, it is found that several known integral operators
carry functions from $W_{\beta}(\alpha,\gamma)$ into $M(\xi)$.
Results for a more generalized operator related to $V_\lambda(f)(z)$ are also given.
\end{abstract}



\maketitle

\pagestyle{myheadings}
\markboth{
Satwanti Devi and A. Swaminathan
}{
Integral transforms of functions
}

\section{Introduction}\label{sec-W-pascu-intro}

        Let $\mathcal{A}$ denote the class of all functions $f$
analytic in the open unit disc ${\mathbb{D}}=\{z\in{\mathbb{C}}: |z|<1\}$ with the normalization $f(0)=f'(0)-1=0 $
and $\mathcal{S}$ be the class of functions $f \in \mathcal{A}$ that are univalent in ${\mathbb{D}}$.
A function $f \in \mathcal{S}$ is said to be starlike $(S^\ast)$ or convex $(C)$, if $f$ maps ${\mathbb{D}}$ conformally onto the domains, respectively,
starlike with respect to origin and convex.
Note that in ${\mathbb{D}}$, if $f\in C \Longleftrightarrow zf^\prime\in S^\ast$ follows from the
well-known Alexander theorem (see \cite{DU} for details).
An useful generalization of the class $S^\ast$ is the class $S^\ast(\sigma)$ that has the analytic
characterisation $S^\ast(\sigma)=\left\{f\in A: {\rm Re \,} \dfrac{zf^\prime}{f}>\sigma; \,0\leq\sigma<1\right\}$ and $S^\ast(0)\equiv S^\ast$.
Various generalization of classes $S^\ast$ and $C$ are abundant in the literature.
One such generalization is the following:

A function $f\in\mathcal{A}$ is said to be in the Pascu class of
$\alpha$ -convex functions of order $\sigma$ if \cite{Pascu}
\begin{align*}
{\rm Re \,}\frac{\alpha z(zf'(z))'+(1-\alpha)zf'(z)}{\alpha zf'(z)+(1-\alpha)f(z)}>\sigma,\quad\quad 0\leq\alpha\leq 1,\quad 0\leq\alpha\leq 1,
\end{align*}
or in other words
\begin{align*}
\alpha zf'(z) +(1-\alpha)f(z) \in\mathcal{S^{\ast}}(\sigma).
\end{align*}
This class is denoted by $M(\alpha, \sigma)$. Even though, this class is known as Pascu class of
$\alpha$ -convex functions of order $\sigma$, since we use the parameter $\alpha$ for another important class, we
denote this class by $M(\xi, \sigma)$, $0\leq\xi\leq 1$, and we remark that, in the sequel, we only consider the class
$M(\xi):=M(\xi,0)$. Clearly $M(0)=S^{\ast}$ and $M(1)=C$ which implies that this class
$M(\xi)$ is a smooth passage between the class of starlike and convex functions.


    The main objective of this work is to find conditions on the non-negative real valued integrable function $\lambda(t)$ satisfying
$\int_0^1\lambda(t) dt=1$, such that the operator
\begin{align}\label{eq-lambda-operator}
F(z)= V_{\lambda}(f)(z):=\int_0^1 \lambda(t) \dfrac{f(tz)}{t}dt
\end{align}
is in the class $M(\xi)$. Note that this operator was introduced in \cite{Four-rus-extremal}.
To investigate this admissibility property the class to which the function $f$ belongs is important.
Let $W_{\beta}(\alpha,\gamma)$, $\alpha\geq 0$, $\gamma\geq 0$ and $\beta<1$,
denote the class of all normalized analytic functions $f$ in the unit disc ${\mathbb{D}}$ such
that
\begin{align*}
{\rm Re\,} \left( e^{i\phi}\left((1-\alpha+2\gamma)\frac{f}{z}+(\alpha-2\gamma)f'+\gamma zf''-\beta\right)\frac{}{}\right)>0, \quad z\in {\mathbb{D}}
\end{align*}
for some $\phi\in {\mathbb{R}}$. This class and its particular cases were considered by many authors so that the corresponding
operator given by $\eqref{eq-lambda-operator}$ is univalent and in $M(\xi)$ for some particular values of $\alpha$, $\beta$, $\gamma$ and $\xi$.
This work was motivated in \cite{Four-rus-extremal} by studying the conditions under which $V_{\lambda}(W_{\beta}(1,0))\subset M(0)$
and generalized in \cite{Kim-Ron} by studying the case $V_{\lambda}(W_{\beta}(\alpha, 0))\subset M(0)$.
Similar situation for the convex case, namely $V_{\lambda}(W_{\beta}(1,0)) \subset M(1)$ was initiated in \cite{Ali}.
After several generalizations by many authors, recently, the conditions under which
$V_{\lambda}(W_{\beta}(\alpha,\gamma)) \subset M(0)$ was obtained in \cite{Abeer S*} and the corresponding results for the convex case
so that $V_{\lambda}(W_{\beta}(\alpha,\gamma)) \subset M(1)$ was obtained in \cite{Mahnaz C}.
Applications involving several well known integral transforms were studied in \cite{Abeer S*} and \cite{Mahnaz C} (see also \cite{Sarika})
For all the literature involving the complete study in this direction so far we refer to
\cite{Abeer S*,Mahnaz C,rag M,Sarika} and references therein.

In this work, we find conditions on $\lambda(t)$ so that $V_{\lambda}(W_{\beta}(\alpha,\gamma)) \subset M(\xi)$ using duality techniques
which are presented in Section $\ref{sec-W-pascu-main-results}$. As applications, in Section $\ref{sec-W-pascu-application}$,
we consider particular values for $\lambda(t)$ in \eqref{eq-lambda-operator} so that results for some of the well-known integral operators can
be deduced. A more generalized operator introduced in \cite{Ali} is considered in Section $\ref{sec-W-pascu-genl-operator}$ for similar type of results.

First we underline some preliminaries that are useful for our
discussion. We introduce two constants $\mu\geq0$ and
$\nu\geq0$ satisfying \cite{Abeer S*,Ali-AAA-3rd}
\begin{equation}\label{eq-mu+nu}
\mu+\nu=\alpha-\gamma \quad \text{and} \quad\mu\nu=\gamma.
\end{equation}

When $\gamma=0$, then $\mu$ is chosen to be $0$, in which case,
$\nu=\alpha \geq 0$. When $\alpha=1+2\gamma$, \eqref{eq-mu+nu} yields
$\mu+\nu=1+\gamma=1+\mu\nu,$ or $(\mu-1)(1-\nu)=0$, and leads to two cases

\begin{itemize}
\item[{\rm{(i)}}] For $\gamma>0,$ then choosing $\mu=1$ gives $\nu=\gamma$.
\item[{\rm{(ii)}}] For $\gamma=0$, then $\mu=0$ and $\nu=\alpha=1$.
\end{itemize}
\begin{note}
Since the case $\gamma=0$ is considered in \cite{rag M}, we only consider results for
the case $\gamma>0$, except for Theorem \ref{app-W-class-gamma0} (see Remark \ref{remark-W-gamma0}).
\end{note}

Next we introduce two known auxiliary functions \cite{Abeer S*}. Let
\begin{equation}\label{phi-series}
\phi_{\mu,\nu}(z)=1+\sum_{n=1}^{\infty }\dfrac{(n\nu +1)(n\mu+1)}{n+1}z^{n},
\end{equation}
and
\begin{eqnarray}\label{psi-series}
\psi_{\mu,\nu}(z)=\phi_{\mu,\nu}^{-1}(z)&=&1+\sum\limits_{n=1}^{\infty }\frac{n+1}{(n\nu+1)(n\mu +1)}z^{n}
=\int_0^1\int_0^1\frac{dsdt}{(1-t^\nu s^\mu z)^2}.
\end{eqnarray}

Here $\phi_{\mu,\nu}^{-1}$ denotes the convolution inverse of $\phi_{\mu,\nu}$
such that $\phi_{\mu,\nu}\ast \phi_{\mu,\nu}^{-1}=1/(1-z)$. By $\ast$, we mean the following:
If $f$ and  $g$ are in ${\mathcal{A}}$ with the power series
expansions $ f(z)=\sum_{k=0}^{\infty} a_{k}z^k$ and
$g(z)= \sum_{k=0}^{\infty} b_{k}z^k$ respectively, then the convolution or Hadamard product
of $f$  and  $g$ is given by $ h(z)= \sum_{k=0}^{\infty}a_{k}b_{k}z^k.$

Since $\nu \geq 0$, $\mu \geq 0$, when $\gamma \geq 0$, making the change of
variables $u =t^{\nu }$, $v =s^{\mu }$ in \eqref{psi-series} result in
writing  $\psi_{\mu,\nu}$ as
\begin{align*}
\psi_{\mu,\nu}(z)=\displaystyle\left\{
\begin{array}{cll}&\displaystyle\dfrac{1}{\mu\nu}\int_0^1\int_0^1\dfrac{u^{1/\nu-1}v^{1/\mu-1}}{(1-uvz)^2}dudv,
\quad \gamma>0,\\ \\
&\displaystyle\int_0^1\dfrac{dt}{(1-t^\alpha z)^2}, \quad \quad \quad \gamma=0,
\alpha\geq0.
\end{array}\right.
\end{align*}

Now let $g$ be the solution of the initial value-problem
\begin{align}\label{de-g}
\frac{d}{dt}t^{1/\nu}(1+g(t))=\left\{
\begin{array}{cll}&\displaystyle
\dfrac{2}{\mu\nu}t^{1/{\nu-1}}\int_0^1 \frac{s^{1/\mu-1}}{(1+st)^2}ds, \quad
\gamma>0,\\\\
&\displaystyle\dfrac{2}{\alpha}\dfrac{t^{1/\alpha-1}}{(1+t)^2}, \quad \quad \quad
\gamma=0, \alpha>0,
\end{array}\right.
\end{align}
satisfying $g(0)=1$. The series solution is given by
\begin{align}\label{series-g(t)}
g(t)=2\sum_{n=0}^{\infty
}\frac{(n+1)(-1)^{n}t^{n}}{(1+\mu n)(1+\nu n)}-1.
\end{align}
\newline
Let $q$ be the solution of the differential equation
\begin{align}\label{de-q}
\dfrac{d}{dt}t^{1/\nu}q(t)=\left\{
\begin{array}{cll}&\displaystyle\dfrac{1}{\mu\nu}
t^{1/{\nu-1}}\int_0^1 {s^{1/\mu-1}}\dfrac{(1-st)}{(1+st)^3}ds, \quad
\gamma>0\\\\
&\displaystyle\dfrac{1}{\alpha}{t^{1/\alpha-1}}\dfrac{(1-t)}{(1+t)^3}, \quad \quad \quad
\gamma=0, \alpha>0.
 \end{array}\right.
\end{align}
satisfying $q(0)=0$. The series solution of $q(t)$ is given by
\begin{align}\label{series-q(t)}
q(t)=\sum_{n=0}^\infty\dfrac{(n+1)^2(-1)^n t^n}{(1+\mu n)(1+\nu n)}
\end{align}
Note that $q(t)$ also satisfies $2q(t)=tg^\prime(t)+g(t)+1$.

Our main results is the generalization of the following results given in
\cite{Abeer S*} and \cite{Mahnaz C}. The necessary and sufficient
conditions under which the operator $V_\lambda(f(z))$ carries the
function $f(z)$ from ${W}_\beta(\alpha, \gamma)$, to the classes
$S^\ast$ and $C$, respectively are given in next two results.
\begin{theorem}\label{abeer-W-class-S*}\cite{Abeer S*}
Consider $\mu\geq 0$, $\nu\geq 0$ given by \eqref{eq-mu+nu} with $\beta < 1$ satisfying
\begin{align}\label{beta-abeer-W-class-S*}
\frac{\beta}{1-\beta}=-\int_0^1\lambda(t)g(t)dt,
\end{align}
where $g$ is the solution of the initial value-problem \eqref{de-g}
and let $f\in \mathcal{W}_\beta(\alpha, \gamma)$.
Assume that $t^{1/\nu}\Lambda_{\nu}(t)\rightarrow 0$, and
$t^{1/\mu}\Pi_{\mu, \nu}(t)\rightarrow 0$ as $t\rightarrow 0^+.$
Then $\ F(z)=V_\lambda(f)(z)$ is in $\mathcal{S^*}$ if and only if
\begin{align}\label{abeer-W-class-S*-result}
\left\{
\begin{array}{cll}&\displaystyle
{\rm Re\,}{\int_0^1\Pi_{\mu,\nu}(t)t^{1/\mu-1}\left(
\frac{h(tz)}{tz}-\frac{1}{(1+t)^2}\right)dt}\geq 0, & \gamma >0,\\
&\displaystyle{\rm Re\,}{\int_0^1\Pi_{0,\alpha}(t)t^{1/\alpha-1}\left(
\frac{h(tz)}{tz}-\frac{1}{(1+t)^2}\right)dt}\geq 0, & \gamma=0,
 \end{array}\right.
\end{align}
where
\begin{align}\label{eqn-lambda-nu}
\Lambda_\nu(t)=\int_t^1\frac{\lambda(x)}{x^{1/\nu}}dx,\ \ \nu>0,
\end{align}

\begin{align}\label{eqn-Pi-nu-mu}
\Pi_{\mu, \,\nu}(t)=\left\{
\begin{array}{cll}&\displaystyle
\int_t^1\Lambda_\nu(x)x^{1/\nu-1-1/\mu}dx,
\quad \gamma>0 (\mu>0,\nu>0),\\
&\displaystyle\Lambda_\alpha(t) , \quad \gamma=0 (\mu=0,
\nu=\alpha>0)
 \end{array}\right.
\end{align}
and
\begin{align*}
h(z)=\frac{z(1+\frac{\epsilon-1}{2}z)}{(1-z)^2}, \quad |\epsilon|=1.
\end{align*}
\end{theorem}

\begin{theorem}\label{Mahnaz-W-class-C}\cite{Mahnaz C}
Let $f\in \mathcal{W}_\beta(\alpha, \gamma)$,  $\mu\geq 0$, $\nu\geq
0$ satisfy $(\ref{eq-mu+nu})$ and $\beta< 1,$ be given by
\begin{align}\label{beta-Mahnaz-W-class-C}
\frac{\beta-1/2}{1-\beta}=-\int_0^1\lambda(t)q(t)dt,
\end{align}
where $q$ is given by $(\ref{de-q})$.
Further $\Lambda_\nu(t)$ and $\Pi_{\nu, \,\mu}(t)$ are given in $\eqref{eqn-lambda-nu}$ and $(\ref{eqn-Pi-nu-mu})$
and assume that $t^{1/\mu}\Lambda_\nu(t)\rightarrow 0$, and
$t^{1/\nu}\Pi_{\mu, \,\nu}(t)\rightarrow 0$ as $t\rightarrow 0^+.$ Then
\begin{align}\label{Mahnaz-W-class-C-result}
\left\{
\begin{array}{cll}&\displaystyle
{\rm Re\,}
{\int_0^1\Pi_{\nu, \,\mu}(t)t^{1/\mu-1}\left(h'(tz)-\frac{1-t}{(1+t)^3}\right)dt}\geq
0, \quad \gamma >0,
\\
&\displaystyle{\rm Re\,}
{\int_0^1\Pi_{0, \,\alpha}(t)t^{1/\alpha-1}\left(h'(tz)-\frac{1-t}{(1+t)^3}\right)dt}\geq
0, \quad \gamma=0,
\end{array}\right.
\end{align}
if and only if $\ F(z)=V_\lambda(f)(z)$ is in $C$.
\end{theorem}

It is difficult to verify conditions \eqref{abeer-W-class-S*-result} and \eqref{Mahnaz-W-class-C-result}.
Hence the following results involving sufficient conditions are useful for finding applications.
\begin{theorem}\label{abeer-W-class-dec-S*}\cite{Abeer S*}
Let
$\Lambda_\nu$ and $\Pi_{\mu, \nu}$ are defined in
$\eqref{eqn-lambda-nu}$ and $\eqref{eqn-Pi-nu-mu}$. Assume that both $\Pi_{\mu, \nu}$ and $\Lambda_\nu$ are
integrable on $\rm[0,1]$ and positive on $\rm(0,1)$. Assume further that
$\mu\geq 1$ and
\begin{equation}\label{abeer-W-class-dec-S*-result}
\frac{\Pi_{\mu, \nu}(t)}{1-t^2}
\end{equation}
is decreasing on $(0,1)$. If $\beta$ satisfies
$\eqref{beta-abeer-W-class-S*}$ and $f\in \mathcal{W}_\beta(\alpha, \gamma)$,
then $V_\lambda(f)\in \mathcal{S^*}$.
\end{theorem}

\begin{theorem}\label{Mahnaz-W-class-dec-C}\cite{Mahnaz C}
Let
$\Lambda_\nu$ and $\Pi_{\mu, \nu}$ are defined in
$\eqref{eqn-lambda-nu}$ and $\eqref{eqn-Pi-nu-mu}$. Assume that both are integrable on $[0,1]$ and
positive on $(0,1)$. Assume further that $\mu\geq 1$ and
\begin{align}\label{Mahnaz-W-class-dec-C-result}
\dfrac{\Lambda_{\nu}(t)t^{1/\nu-1/\mu}+(1-1/\mu)\Pi_{\mu,\nu}(t)}{1-t^2},
\end{align}
is decreasing on $(0,1)$. If $\beta$ satisfies $(\ref{beta-Mahnaz-W-class-C})$
and $f\in\mathcal{W}_{\beta}(\alpha,\gamma)$, then $V_{\lambda}(f)\in C$.
\end{theorem}

\section{Main Results}\label{sec-W-pascu-main-results}

We start with a result that gives both necessary and
sufficient condition for an integral transform that satisfies the admissibility property of the class
$W_{\beta}(\alpha,\gamma)$, which contain non-univalent functions also,
to the Pascu class $M(\xi)$.
\begin{theorem}\label{thm-W-Pascu-dual-equiv}
Let $\mu>0$ , $\nu>0$, satisfies \eqref{eq-mu+nu} and $\beta <1$ satisfies
\begin{align}\label{W-class-beta-cond}
\dfrac{\beta}{(1-\beta)}=-\int_0^1 \lambda(t)[(1-\xi)g(t)+\xi(2q(t)-1)]dt,
\end{align}
where $g(t)$ and $q(t)$ are defined by the differential equations
given in \eqref{de-g} and \eqref{de-q} respectively. Assume that $t^{1/\nu}\Lambda_\nu (t)\rightarrow 0$ and
$t^{1/\mu}\Pi_{\mu, \,\nu}(t)\rightarrow 0$ as $t\rightarrow 0^+$. Then
\begin{align*}
N_{\Pi_{\mu, \,\nu}}\geq 0\Longleftrightarrow F=V_{\lambda}(W_{\beta}(\alpha,\gamma))\in M(\xi)
\end{align*}
or
\begin{align*}
\xi zF^\prime + (1-\xi)F \in S^\ast,
\end{align*}
where
\begin{equation}\label{eqn-W-pascu-N-Pi}
N_{\Pi_{\mu,\,\nu}}(h)=\inf_{z\in\Delta} \int_{0} ^{1} t^{{1/{\mu}} -1} \Pi_{\mu, \,\nu}(t){\mathcal L}_{\xi, \,z}(t)dt
\end{equation}
and
\begin{equation*}
{\mathcal L}_{\xi,\,z}(t)=(1-\xi) \left( {\rm Re \,} \dfrac{h(tz)}{tz} - \dfrac{1}{(1+t)^2} \right) +
\xi \left( {\rm Re \,}{ \,} h^{\prime}(tz) - \dfrac{(1-t)}{(1+t)^3}\right).
\end{equation*}
The value of $\beta$ is sharp.
\end{theorem}
\begin{proof}
Let
\begin{align*}
H(z)=(1-\alpha+2\gamma)\dfrac{f(z)}{z}+(\alpha-2\gamma)f^\prime (z) + \gamma z f^{\prime\prime}(z) \quad {\mbox{and}}
\quad G(z)=\dfrac{(H(z)-\beta)}{1-\beta}.
\end{align*}
Since ${\rm Re \,}{ \ } e^{i\phi} G(z)>0$, \cite{Rus}
we may assume that $G(z)=\dfrac{1+xz}{1+yz}$, $\mid{x}\mid={}\mid{y}\mid=1$.
Further, assuming $f(z)=z+\sum_{n=2}^{\infty} a_n z^n$, from \eqref{eq-mu+nu} we get
\begin{align*}
&&H(z)& =\left(1+\mu \nu -\nu -\mu\right)\frac{f(z)}{z}+\left(\nu +\mu
-\mu \nu\right)f'(z)+\mu \nu zf''(z) \\
&&&=\mu \nu z^{1-1/\mu }\frac{d}{dz}\left[
z^{1/\mu -1/\nu +1}\frac{d}{dz}\left( z^{1/\nu -1}f(z)\right)
\right] =f^\prime (z) \ast \phi_{\mu, \,\nu}(z)\\
\Longrightarrow && f^\prime (z)&=H(z) \ast \psi_{\mu, \,\nu} =\left[(1-\beta)\left(\dfrac{1+xz}{1+yz}\right)+\beta\right]\ast \psi_{\mu, \,\nu},
\end{align*}
using \eqref{phi-series} and \eqref{psi-series}.
This gives
\begin{align}\label{f-M}
\dfrac{f(z)}{z}=\dfrac{1}{z}\int_0^z \left[(1-\beta)\left(\dfrac{1+x\omega}{1+y\omega}\right)+\beta\right]d\omega \ast \psi_{\mu, \,\nu}.
\end{align}
Since $F\in M(\xi)$, implies $\xi zF^\prime + (1-\xi)F \in S^\ast$,
by the well known result \cite[p.94]{Rus} of convolution theory we get
\begin{align*}
\xi zF^\prime + (1-\xi)F \in S^\ast \quad {\mbox{if and only if\,}} \quad  0\neq \dfrac{1}{z}\left[((1-\xi)F + \xi zF^\prime) \ast h(z) \right],
\quad \quad |z|<1.
\end{align*}
Thus
\begin{align*}
&0\neq\left((1-\xi)\dfrac{F}{z}+ \xi F' \right)\ast\dfrac{h(z)}{z}\\\\
&= (1-\xi)\int_0^1\dfrac{\lambda(t)}{1-tz}\ast\dfrac{f(z)}{z}\ast\dfrac{h(z)}{z}+\xi\int_0^1\dfrac{\lambda(t)}{1-tz}\ast f^\prime(z)\ast\dfrac{h(z)}{z}\\
&=(1-\xi)\int_0^1\lambda(t)\dfrac{h(tz)}{tz}dt \ast \left[\dfrac{1}{z}\int_0^z(1-\beta)\left(\dfrac{1+x\omega}{1+y\omega}\right)d\omega+\beta\right]\\
& \quad\quad\quad +\xi\int_0^1 \lambda(t) h^\prime(tz)dt\ast \left[\dfrac{1}{z}\int_0^z(1-\beta)\left(\dfrac{1+x\omega}{1+y\omega}\right)d\omega+\beta\right]\\
&=(1-\beta)\left[(1-\xi)\int_0^1\lambda(t)\left(\dfrac{1}{z}\int_0^z\dfrac{h(t\omega)}{t\omega}d\omega+\dfrac{\beta}{1-\beta}\right)dt\right. \\
&\quad\quad\left.+\xi\int_0^1\lambda(t)\left(\dfrac{1}{z}\int_0^z h^\prime(t\omega)d\omega+\dfrac{\beta}{1-\beta}\right)dt\right]
\ast\psi(z)\ast\dfrac{1+xz}{1+yz}.
\end{align*}
This condition holds if and only if \cite[p.23]{Rus}
\begin{align*}
{\rm Re \,}\left((1-\beta)\left[(1-\xi)\int_0^1\lambda(t)\left(\dfrac{1}{z}\int_0^z\dfrac{h(t\omega)}{t\omega}d\omega+\dfrac{\beta}{1-\beta}\right)dt\right.\right. \\
\left.\left.+\xi\int_0^1\lambda(t)\left(\dfrac{1}{z}\int_0^z h^\prime(t\omega)d\omega+\dfrac{\beta}{1-\beta}\right)dt\right]
\ast\psi(z)\right)>\dfrac{1}{2}.
\end{align*}
Using the results \eqref{abeer-W-class-S*-result} and\eqref{Mahnaz-W-class-C-result} from
Theorem \ref{abeer-W-class-S*} and Theorem \ref{Mahnaz-W-class-C} respectively, we obtain
\begin{equation*}
\int_{0} ^{1} t^{{1/{\mu}} -1} \Pi_{\mu, \,\nu}(t)\left[(1-\xi) \left( {\rm Re \,} \dfrac{h(tz)}{tz} - \dfrac{1}{(1+t)^2} \right) +
\xi \left( {\rm Re \,} h^{\prime}(tz) - \dfrac{(1-t)}{(1+t)^3}\right)\right]dt\geq 0
\end{equation*}
which is the required result.

To verify sharpness, let $W_\beta(\alpha, \,\gamma)$ be the solution
of the differential equation
\begin{align*}
(1-\alpha+2\gamma)\dfrac{f(z)}{z}+(\alpha-2\gamma)f^\prime (z) + \gamma z f^{\prime\prime}(z)=\beta+(1-\beta)\dfrac{1+z}{1-z}
\end{align*}
where $\beta<\beta_0$ satisfies \eqref{W-class-beta-cond}. Further simplification using \eqref{series-g(t)} and \eqref{series-q(t)} gives
\begin{align*}
\dfrac{\beta_0}{1-\beta_0}=-1-\left[(1-\xi)\sum_{n=1}^\infty\dfrac{2(n+1)(-1)^n\tau_n}{(1+\mu n)(1+\nu n)}
+\xi\sum_{n=1}^\infty\dfrac{2(n+1)^2(-1)^n\tau_n}{(1+\mu n)(1+\nu n)}\right],
\end{align*}
where $\displaystyle\tau_n=\int_0^1\lambda(t)t^n dt$.

Clearly $F=V_\lambda(f(z))\in M(\xi)\Longrightarrow K(z):=\xi zF^\prime+(1-\xi)F\in S^\ast$.
Using $f^\prime (z)=H(z) \ast \psi_{\mu, \,\nu}$ and the series expansion of $f(z)$, we get
\begin{align*}
f(z)=z+\sum_{n=1}^{\infty }\dfrac{2(1-\beta )}{(n\mu +1)(n\nu+1)}z^{n+1}.
\end{align*}

This means
\begin{align}\nonumber
F=V_\lambda(f(z))& 
=z+\sum_{n=1}^\infty\dfrac{2(1-\beta)}{(n\mu+1)(n\nu+1)}\left(\int_0^1\lambda(t)t^n dt\right)z^{n+1}\\
\label{series-v(lambda)}&=z+\sum_{n=1}^\infty\dfrac{2(1-\beta)\tau_n}{(n\mu+1)(n\nu+1)}z^{n+1}.
\end{align}
Using \eqref{series-v(lambda)}, a simple computation gives
\begin{align*}
K(z)=(1-\xi)\left(z+\sum_{n=1}^\infty\dfrac{2(1-\beta)\tau_n z^{n+1}}{(n\mu+1)(n\nu+1)}\right)+
\xi\left(z+\sum_{n=1}^\infty\dfrac{2(1-\beta)(n+1)\tau_n z^{n+1}}{(n\mu+1)(n\nu+1)}\right).
\end{align*}
This means
\begin{align*}
zK^\prime(z)|_{z=-1}&=1+\xi\sum_{n=1}^\infty\dfrac{2(1-\beta)(n+1)^2\tau_n (-1)^n}{(n\mu+1)(n\nu+1)}+
(1-\xi)\sum_{n=1}^\infty\dfrac{2(1-\beta)(n+1)\tau_n (-1)^n}{(n\mu+1)(n\nu+1)}\\
&=1-\dfrac{(1-\beta)}{1-\beta_0}<0.
\end{align*}
Hence $zK^\prime(z)=0$ for some $z\in \mathbb{D}$, so $K(z)$ is not even locally univalent
in $\mathbb{D}$. This shows that the result is sharp for $\beta$.
\end{proof}

\begin{remark}
This result generalizes various results known in this direction. For example,\\
$\xi=0$ gives Theorem $\ref{abeer-W-class-S*}$ \cite[Theorem $3.1$]{Abeer S*} and $\xi=1$ gives
Theorem $\ref{Mahnaz-W-class-C}$ \cite[Theorem $3.1$]{Mahnaz C}.
For other particular cases with $\xi =0$ or $1$, we refer to \cite{Abeer S*,Mahnaz C} and references therein.
\end{remark}

\begin{theorem}\label{A.S-result-inc}
Let $\Pi_{\mu, \,\nu}$ and $\Lambda_\nu $ be defined as in $(\ref{eqn-lambda-nu})$ and $(\ref{eqn-Pi-nu-mu})$ respectively, with both of them
integrable on $[0,1]$ and positive on $(0,1)$. Further assume that  $0 \leq \xi \leq 1$, $\mu \geq 1$ and
\begin{equation}\label{A.S-res-inc-main-eq}
\dfrac{\xi t^{{1 / {\xi}} -{1 / {\mu}} +1}\hspace{.2cm}d {\left(t^{{1 / {\mu}} - {1 / {\xi}}} \Pi_{\mu, \,\nu}(t)\right)}}{(1-t^2)}
\end{equation}
is increasing on $(0,1)$. Then for $\beta$ satisfying $\eqref{W-class-beta-cond}$, $V_{\lambda}(W_{\beta}(\alpha,\gamma))\in M(\xi)$.
\end{theorem}
\begin{proof}
Consider
\begin{equation*}
N_{\Pi_{\mu, \,\nu}}(h)=\int_0^1 t^{{1/{\mu}} -1} \Pi_{\mu, \nu}(t)\left[(1-\xi)\left(\dfrac{h(tz)}{tz} - \dfrac {1}{(1+t)^2} \right)+
\xi\left(h^{\prime} (tz) - \dfrac{(1-t)}{(1+t)^3}\right)\right]dt
\end{equation*}

\begin{equation*}
=\int_{0}^{1}t^{{1/{\mu}}-1}\Pi_{\mu, \,\nu}(t)\left[(1-\xi)\left(\dfrac{h(tz)}{tz}-\dfrac{1}{(1+t)^2}\right)+
\xi\dfrac{d}{dt}\left(\dfrac{h(tz)}{z}-\dfrac{t}{(1+t)^2}\right)\right]dt.
\end{equation*}
Integration by parts gives
\newline
$\displaystyle
N_{\Pi_{\mu, \,\nu}}(h)=(1-\xi)\int_0^1 t^{{1/{\mu}} -1} \Pi_{\mu, \nu}(t)\left(\dfrac{h(tz)}{tz}-\dfrac {1}{(1+t)^2} \right)dt
$
\begin{equation*}
-
\xi\int_0^1\dfrac{d}{dt}\left(t^{{1/{\mu}}-1}\Pi_{\mu, \,\nu}(t)\right)\left(\dfrac{h(tz)}{z}-\dfrac{t}{(1+t)^2}\right)dt
\end{equation*}
\begin{equation*}
=\int_0^1\left(\left[(1-\xi)-\xi(\dfrac{1}{\mu}-1) \right] \Pi_{\mu, \,\nu}(t)t^{{1/\mu}-1}-\xi\Pi_{\mu, \,\nu}^\prime (t)t^{1/\mu} \right)
\left(\dfrac{h(tz)}{tz}-\dfrac {1}{(1+t)^2} \right)dt.
\end{equation*}
\begin{equation*}
=\int_0^1 t^{{1/\mu}-1} \left((1-\dfrac{\xi}{\mu})\Pi_{\mu, \,\nu}(t)-\xi t \Pi_{\mu, \,\nu}^\prime (t)\right)
\left(\dfrac{h(tz)}{tz}-\dfrac {1}{(1+t)^2} \right)dt,
\end{equation*}
by a simple computation.
It is easy to see that, from Theorem \eqref{abeer-W-class-dec-S*} and \eqref{Mahnaz-W-class-dec-C},
$N_{\Pi_{\mu, \,\nu}}(h)\geq 0$ only when $\xi\dfrac{\Pi_{\mu, \nu}(t)}{1-t^2}
+ (1-\xi) \dfrac{\Lambda_{\nu}(t)t^{1/\nu-1/\mu}+(1-1/\mu)\Pi_{\mu,\nu}(t)}{1-t^2}$
is decreasing on $(0,1)$ which is nothing but \eqref{A.S-res-inc-main-eq} and the proof is complete
by applying Theorem $\ref{thm-W-Pascu-dual-equiv}$.
\end{proof}

\begin{remark}
Even though, we did not consider the case $\gamma=0$, even at $\gamma=0$, Theorem \ref{A.S-result-inc}
does not reduces to a similar result given in \cite{rag M}. This is due to the fact that,
our condition \eqref{A.S-res-inc-main-eq} has the term $(1-t^2)$ in the denominator, whereas
the corresponding result in \cite{rag M} has the term $log(1/t)$ and hence has different condition.
\end{remark}

\section{Applications}\label{sec-W-pascu-application}
It is difficult to check the condition given in Section 2, for $V_{\lambda}(W_{\beta}(\alpha,\gamma))\subset M(\xi)$.
In order to find applications, simplified conditions are required. For this purpose, from \eqref{A.S-res-inc-main-eq},
it is enough to show that
\begin{align*}
\dfrac{\xi t^{{1 / {\xi}} -{1 / {\mu}} +1}\hspace{.2cm}d
{\left(t^{{1 / {\mu}} - {1 / {\xi}}} \Pi_{\mu, \,\nu}(t)\right)}}{(1-t^2)}
\end{align*}
is increasing on (0,1) which is equivalent of having
\begin{align*}
g(t)=\dfrac{\left(1-\dfrac{\xi}{\mu}\right)\Pi_{\mu, \,\nu}(t)+\xi t^{1/\nu - 1/\mu}\Lambda_\nu(t)}{(1-t^2)}
\end{align*}
is decreasing on (0,1), where $\Lambda_\nu(t)$ and $\Pi_{\mu, \,\nu}(t)$ are defined in
$\eqref{eqn-lambda-nu}$ and $\eqref{eqn-Pi-nu-mu}$. It is enough to have $g^\prime(t)\leq 0$.\\
Let $g(t)=p(t)/(1-t^2)$ where $p(t)=\left(1-\dfrac{\xi}{\mu}\right)\Pi_{\mu, \,\nu}(t)+\xi t^{1/\nu - 1/\mu}\Lambda_\nu(t)$.
So to satisfy the above condition we need to have
\begin{align*}
L(t)=p(t)+\dfrac{(1-t^2)p^\prime(t)}{2t}\leq 0.
\end{align*}
Since $\Lambda_\nu(1)=0$ and $\Pi_{\mu, \,\nu}(1)=0$ we get
 $L(1)=0$. This implies that it suffies to have $L(t)$ is increasing on (0,1), which means
\begin{align*}
L^\prime(t)=\dfrac{(1-t^2)}{2t^2}[tp^{\prime\prime}(z)-p^\prime(z)]\geq 0.
\end{align*}
The above equation also holds if $tp^{\prime\prime}(z)-p^\prime(z)\geq 0$ which is equivalent to the condition
\begin{align}\label{initial-cond-W-class}
\left(\dfrac{\xi}{\nu}-1\right)\left(\dfrac{1}{\nu}-\dfrac{1}{\mu}-2\right)\Lambda_\nu (t)+
\left(1+\xi\left[1+\dfrac{1}{\mu}-\dfrac{1}{\nu}\right]\right)t^{1-1/\nu}\lambda(t)-
\xi t^{2-1/\nu}\lambda^\prime(t)\geq 0.
\end{align}
This inequality can further be reduced to
\begin{align}\label{final-cond-W-class}
(1-\xi)\left[\left(1+\dfrac{1}{\mu}\right)\lambda(t)-t\lambda^\prime(t)\right]+\xi\left[t^2\lambda^{\prime\prime}(t)-\dfrac{1}{\mu}t\lambda^\prime(t)\right]\geq 0
\end{align}
if the inequality
\begin{align}\label{additional-cond-W-class}
\xi\dfrac{\lambda^\prime(1)}{\lambda(1)}\leq 1+\xi\left(1+\dfrac{1}{\mu}-\dfrac{1}{\nu}\right)
\end{align}
is true. So, in order to obtain further results, we check conditions \eqref{final-cond-W-class} and \eqref{additional-cond-W-class}.
Note that, whenever $\lambda(1)=\lambda^\prime(1)=0$, from \eqref{initial-cond-W-class} we see that it is sufficient to check
\eqref{final-cond-W-class} as there is no necessity for the condition given by \eqref{additional-cond-W-class}.

As the first application, we consider $\lambda(t)=(c+1)t^c, c>-1$, so that the Bernardi operator of function in $W_{\beta}(\alpha,\gamma)$
is in $M(\xi)$.
\begin{theorem}\label{application1}
Let $0\leq \xi\leq 1$, $\nu\geq\mu\geq 1$ and $\beta<1$ satisfies \eqref{W-class-beta-cond}. If $f(z)\in W_{\beta}(\alpha,\gamma)$, then the function,
given by the Bernardi operator,
\begin{align*}
V_\lambda(f)(z)=(1+c)\int_0^1 t^{c-1}f(tz)dt
\end{align*}
belongs to $M(\xi)$ if
\begin{align}\label{application-1-cond}
-1<c\leq min\left[\left(1+\dfrac{1}{\mu}-\dfrac{1}{\nu}\right), \left(\dfrac{1+\dfrac{1}{\mu}-\xi}{1+2\xi}\right)\right].
\end{align}
\end{theorem}
\begin{proof}
With $\lambda(t)=(1+c)t^c$, using \eqref{application-1-cond}, we get
\begin{align*}
\xi\dfrac{\lambda^\prime(1)}{\lambda(1)}=\xi c \quad\leq\xi\left(1+\dfrac{1}{\mu}-\dfrac{1}{\nu}\right)
\quad\leq 1+\xi\left(1+\dfrac{1}{\mu}-\dfrac{1}{\nu}\right).
\end{align*}
So the inequality (\ref{additional-cond-W-class}) is satisfied. If the inequality (\ref{final-cond-W-class}) holds then $V_\lambda(f)(z)\in M(\xi)$
which means
\begin{align*}
(1-\xi)\left[\left(1+\dfrac{1}{\mu}\right)(c+1)-c(c+1)\right]+\xi\left[(c-1)c(c+1)-\dfrac{1}{\mu}c(c+1)\right]\geq 0.
\end{align*}
On further simplification this inequality reduces to
\begin{align}\label{M-class-eq1}
1+\dfrac{1}{\mu}-c+c^2\xi\geq 0.
\end{align}
Clearly $(c+1)^2> 0$. Hence using $c^2>-(2c+1)$ and substituting in (\ref{M-class-eq1}), we get
\begin{align*}
\left(c^2\xi+1+\dfrac{1}{\mu}-c\right) > \left(-c(2\xi+1)+1+\dfrac{1}{\mu}-\xi\right),
\end{align*}
which is true by the hypothesis.
\end{proof}
\begin{remark}
\begin{enumerate}\item[]
\item When $\xi=0$ then $-1<c\leq min\left[\left(1+\dfrac{1}{\mu}-\dfrac{1}{\nu}\right),\left(1+\dfrac{1}{\mu}\right)\right]=
\left(1+\dfrac{1}{\mu}-\dfrac{1}{\nu}\right)$
whereas in \cite{Abeer S*}, the range for $c$ is given as $-1<c\leq \left(1+\dfrac{1}{\mu}\right)$.
\item For $\xi=1$, we have $-1<c\leq min\left[\left(1+\dfrac{1}{\mu}-\dfrac{1}{\nu}\right),\left(\dfrac{1}{3\mu}\right)\right]=
\left(\dfrac{1}{3\mu}\right)$.
Result obtained in \cite{Mahnaz C}, for $\nu\geq\mu\geq 1$ is $-1<c\leq \left(2+\dfrac{1}{\mu}-\dfrac{1}{\nu}\right)$.
\end{enumerate}
\end{remark}

So in both the cases the result (Theorem \ref{application1}) obtained for admissibility property of Bernardi operator in $M(\xi)$ class differs from
\cite[Theorem 5.1]{Abeer S*} and \cite[Theorem 5.1]{Mahnaz C}. But our result
is true for  $0<\xi<1$ also.
\begin{remark}\label{remark-W-gamma0}
We recall that, throughout this paper we use $\gamma>0$, as $\gamma=0$ case is considered in
\cite{rag M}. But we note that, for the Bernardi operator, the result given in \cite{rag M} uses the fact
$\lambda(1)=0$, which is not true. Hence , in order to make completion of the work for $W_{\beta}(\alpha,\gamma)$ in this direction we give
the result related to Bernardi operator for $\gamma=0$. Since the condition for $\lambda(t)$, given in \cite{rag M} is different
from the one given by \eqref{additional-cond-W-class} and \eqref{final-cond-W-class}, we explicitly prove this result.
\end{remark}
\begin{theorem}\label{app-W-class-gamma0}
Let $0< \xi< 1$, $\mu\geq 1$ and $\beta<1$ satisfies \eqref{W-class-beta-cond}. If $f(z)\in W_{\beta}(\alpha,\gamma)$, then the function
\begin{align}\label{defn-operator-Bernardi}
V_\lambda(f)(z)=(1+c)\int_0^1 t^{c-1}f(tz)dt
\end{align}
belongs to $M(\xi)$ if
$c>1+\dfrac{1}{\alpha}$ and $\xi\geq\alpha$
for $\gamma=0$.
\end{theorem}
\begin{proof}
Note that for the Bernardi operator, the case $\xi=0$ is considered in \cite[Theorem. 5.1]{Abeer S*}
 and $\xi=1$ is considered in \cite[Theorem. 5.1]{Mahnaz C}.
We consider only the case $0<\xi<1$.\\
For the case $\gamma=0$, the integral operator $V_\lambda(W_\beta(\alpha, \, \gamma))$ is in $M(\xi)$ if and only if
\begin{align*}
\dfrac{\left(1-\dfrac{\xi}{\alpha}\right)\Pi_{0, \,\alpha}(t)+\xi t^{1-1/\alpha}\lambda(t)}{(1-t^2)}
\quad \quad {\mbox{is decreasing on}} \quad (0,1),
\end{align*}
which can be easily obtained as in Theorem \ref{A.S-result-inc}.\\
Let $J(t)=\dfrac{p(t)}{(1-t^2)}$, where
\begin{align*}
p(t)=\left(1-\dfrac{\xi}{\alpha}\right)\Pi_{0, \,\alpha}(t)+\xi t^{1-1/\alpha}\lambda(t)
\end{align*}
for the integral operator to be in $M(\xi)$, $J^\prime(t)\leq 0$ \,
$\Longrightarrow p(t)+\dfrac{(1-t^2)}{2t}p^\prime(t)\leq 0$ \\
Substituting the value of $p(t)$ and $p^\prime(t)$ in the above
equation, we have
\newline
$
J(t)
$
\begin{align}\label{initial-inequality-W-gamma0}
=\left(1-\dfrac{\xi}{\alpha}\right)\Pi_{0, \,\alpha}(t)+\dfrac{(\xi+1)}{2}t^{1-1/\alpha}\lambda(t)+
\dfrac{(\xi-1)}{2}t^{-1-1/\alpha}\lambda(t)+\dfrac{\xi}{2}t^{-1/\alpha}(1-t^2)\lambda^\prime(t)\leq 0
\end{align}
Consider the case when $\lambda(t)=(c+1)t^c$, where $c>-1$.
So the inequality \eqref{initial-inequality-W-gamma0} on simplification reduces to
\begin{align}\label{W-gamma0-J0}\nonumber
J(t)=& \left(1-\dfrac{\xi}{\alpha}\right)\dfrac{c+1}{c-\dfrac{1}{\alpha}+1}\left[1-t^{c-1/\alpha+1}\right]
+\dfrac{1}{2}(c+1)(\xi+1-c\xi)t^{c+1-1/\alpha}\\
& +\dfrac{1}{2}(c+1)(\xi-1+c\xi)t^{c-1-1/\alpha}\leq 0.
\end{align}
If $c>1+\dfrac{1}{\alpha}$ and $\alpha\leq\xi$, then for $t=0$, \eqref{W-gamma0-J0} is satisfied.\\
For $c>1+\dfrac{1}{\alpha}$, $J(0)\leq 0$ which clearly implies that if $J(t)$
is decreasing on $t\in (0,1)$ i.e., $J^\prime(t)\leq 0$, then also the inequality \eqref{W-gamma0-J0} holds.
Further simplifying \eqref{initial-inequality-W-gamma0}, for Bernardi operator, we need to show that
\begin{align}\label{reduced-W-Bernardi-gamma0(1)}
\dfrac{1}{2}(1-\xi)\left(1+\dfrac{1}{\alpha}\right)\lambda(t)+\dfrac{1}{2}\left[\xi\left(1-\dfrac{1}{\alpha}\right)-1\right]t\lambda^\prime(t)+
\dfrac{\xi}{2}t^2\lambda^{\prime\prime}(t)\leq 0.
\end{align}
On substituting the values of $\lambda(t)$, $\lambda^\prime(t)$ and $\lambda^{\prime\prime}(t)$ in \eqref{reduced-W-Bernardi-gamma0(1)},
we get
\begin{align}\label{reduced-W-Bernardi-gamma0(2)}
\left(1-c-\dfrac{1}{\alpha}\right)-\xi\left(1-c^2-\dfrac{1}{\alpha}+\dfrac{c}{\alpha}\right)\leq 0.
\end{align}
To satisfy \eqref{reduced-W-Bernardi-gamma0(2)}, $\xi$ should hold
\begin{align*}
\xi\geq\dfrac{\alpha(1-c)-1}{(c-1)+\alpha(1-c^2)}.
\end{align*}
Since $\alpha\leq\xi$, we get
\begin{align*}
\xi\geq \max\left(\alpha, \,\dfrac{\alpha(1-c)-1}{(c-1)+\alpha(1-c^2)}\right)=\alpha
\end{align*}
satisfying the hypothesis of theorem which gives $V_{\lambda}(f)(z)$ given by \eqref{defn-operator-Bernardi}
is in $M(\xi)$.
\end{proof}

\begin{theorem}
Let $0\leq\xi\leq 1$ and $\nu\geq\mu\geq 1$. If $F\in \mathcal{A}$ satisfies,
\begin{align*}
{\rm Re \,}(F^\prime(z)+\alpha zF^{\prime\prime}(z)+\gamma z^2 F^{\prime\prime\prime}(z))>\beta
\end{align*}
in $\mathbb{D}$, and $\beta<1$ satisfies
\begin{align*}
\dfrac{\beta}{(1-\beta)}=-\int_0^1 \lambda(t)[(1-\xi)g(t)+\xi(2 q(t)-1)]dt,
\end{align*}
where $g(t)$ and $q(t)$ are given by \eqref{series-g(t)} and \eqref{series-q(t)} respectively.
Then $F\in M(\xi)$.
\end{theorem}
\begin{proof}
Let $f(z)=zF^\prime(z)$, then $f\in W_\beta(\alpha, \,\gamma)$. Therefore
\begin{align*}
F(z)=\int_0^1\dfrac{f(tz)}{t}dt.
\end{align*}
When $c=0$, the hypothesis $\nu\geq\mu\geq 1$ satisfies \eqref{application-1-cond} and hence from Theorem \ref{application1},
the required result follows.
\end{proof}

\begin{example}
If $\gamma=1$, $\alpha=3$, then $\mu=1=\nu$. In this case, \eqref{series-g(t)} and \eqref{series-q(t)}
yield
\begin{align*}
\dfrac{\beta}{(1-\beta)}&=2(1-\xi)\int_0^{-1}\dfrac{\log(1-t)}{t}dt-2\xi\log 2+1\\
&=1-2(1-\xi)\left(\dfrac{\pi^2}{12}\right)-2\xi\log 2.
\end{align*}
Thus
${\rm{Re \,}}\left(f^\prime(z)+3zf^{\prime\prime}(z)+z^2f^{\prime\prime\prime}(z)\right)>\beta\Longrightarrow f\in M(\xi)$.
\end{example}
\begin{remark}
\begin{enumerate}\item[]
\item for $\xi=0$ \cite[Remark 5.2]{Abeer S*}, we get $\beta=-1.816378$, such that $f\in M(0)$.
\item For $\xi=1$ \cite[Example 5.2]{Mahnaz C}, we get $\beta=-0.629445$, such that $f\in M(1)$.
\end{enumerate}
\end{remark}

\begin{theorem}\label{application2}
Let $0\leq \xi\leq 1$, $\mu\geq 1$, $B<1$ and $\beta<1$ satisfies \eqref{beta-Mahnaz-W-class-C}. If $f(z)\in W_{\beta}(\alpha,\gamma)$, then the function
\begin{align*}
V_\lambda(f)(z)=k\int_0^1 t^{B-1}(1-t)^{C-A-B}\phi(1-t)\dfrac{f(tz)}{t}dt
\end{align*}
belongs to $M(\xi)$ if
\begin{align*}
B<min\left(\left[2+\dfrac{1}{\mu}\right], (C-A-1)\right).
\end{align*}
\end{theorem}
\begin{proof}
For $V_\lambda(f)(z)$ to be in $M(\xi)$, $\lambda(t)$ should satisfy \eqref{initial-cond-W-class}, where \\
$\lambda(t)=kt^{B-1}(1-t)^{C-A-B}\phi(1-t)$ and \\
$
\displaystyle
\lambda^\prime(t)
$
\begin{align*}
=Kt^{B-2}(1-t)^{C-A-B-1}\left[\left((B-1)(1-t)-(C-A-B)t\right)\phi(1-t)-t(1-t)\phi'(1-t)\right],
\end{align*}
Since $C-A-B>1$, $\lambda(1)=0$ and $\lambda^\prime(1)=0$, generates that it is enough to check
(\ref{final-cond-W-class}).
If we substitute the values of $\lambda(t)$, $\lambda^\prime(t)$ and $\lambda^{\prime\prime}(t)$  in
$\eqref{final-cond-W-class}$, where
\begin{align*}
\intertext{and}
\lambda^{\prime\prime}(t)=&Kt^{B-3}(1-t)^{C-A-B-2}[((B-1)(B-2)(1-t)^2\\
&-2(B-1)(C-A-B)t(1-t)+(C-A-B)(C-A-B-1)t^2)\phi(1-t)\\
&+[2(C-A-B)t-2(B-1)(1-t)]t(1-t)\phi'(1-t)+t^2(1-t)^2\phi"(1-t)].
\end{align*}
a simple computation shows that \eqref{final-cond-W-class} is equivalent to
\begin{align*}
(1-\xi)h_1(t)+\xi h_2(t)\geq 0,
\end{align*}
where
\begin{align*}
h_1(t)=k t^{B-1}(1-t)^{C-A-B-1}(X_1 (t)\phi(1-t)-t(1-t)\phi^\prime(1-t))
\end{align*}
with
\begin{align*}
X_1 (t)=\left[\left(1+\dfrac{1}{\mu}\right)(1-t)-(B-1)(1-t)+(C-A-B)t\right]
\end{align*}
and
\begin{align*}
h_2(t)=k t^{B-1}(1-t)^{C-A-B-2}(X_2 (t)\phi(1-t)+X_3 (t)\phi^\prime(1-t)+X_4 (t)\phi^{\prime\prime}(1-t))
\end{align*}
with
\begin{align*}
 X_2 (t)=&(B-1)(1-t)^2\left[B-2-\dfrac{1}{\mu}\right]+(C-A-B)t(1-t)\left[\dfrac{1}{\mu}-2(B-1)\right] \\
&+(C-A-B)(C-A-B-1)t^2,
\end{align*}
$\hspace{1cm} X_3 (t)=[2(C-A-B)t-2(B-1)(1-t)]t(1-t)+\dfrac{1}{\mu}t(1-t)^2$,\\
and
$X_4 (t)=t^2(1-t)^2.$

Since $\phi(1-t)>0$, $\phi^\prime(1-t)>0$ and $\phi^{\prime\prime}(1-t)>0$
(see also \cite{Abeer S*,Mahnaz C,rag M}), proving $X_i (t)>0$ for $i=1,2,3$ gives
$h_1(t)\geq 0$ and $h_2(t)\geq 0$ which will imply the required result for $0 < t < 1$.

Now $ X_1 (t)=\left[(C-A-B)t+\left(1+\dfrac{1}{\mu}-(B-1)\right)(1-t)\right]>0$, which clearly holds since
$(C-A-B)>1$ and $B<2+\dfrac{1}{\mu}$ for all $t\in(0,1)$.

Similarly to prove $X_2(t)>0$, it is enough to prove
\begin{align*}
(B-1)\left(B-2-\dfrac{1}{\mu}\right)(1-t)+(C-A-B)\left(\dfrac{1}{\mu}-2(B-1)\right)t>0,
\end{align*}
as the other term in $X_2(t)$ is positive on $0<t<1$. Since
$B<1$ and $\dfrac{1}{\mu}>(B-2)$ from the hypothesis, the term involving $(1-t)$ is non-negative.
Given that $\dfrac{1}{\mu}>(B-2)$, since $B<1$ we have $\dfrac{1}{\mu}>2(B-1)$ and $(C-A-B)>1$, which gives
$X_2 (t)>0$ for $0<t<1$.

Now proving $X_3(t)>0$ is equivalent to prove
$2(C-A-B)t+\left[\dfrac{1}{\mu}-2(B-1)\right](1-t)>0$.\\
Since $2(C-A-B)t+\left[\dfrac{1}{\mu}-2(B-1)\right](1-t)>0$,
by hypothesis, and $\dfrac{1}{\mu}t(1-t)^2>0$ for $\mu>0$ and $0<t<1$, $X_3(t)>0$.
\end{proof}
\begin{remark}
\begin{enumerate}\item[]
\item For the particular value of $\xi=0$, Theorem \ref{application2} yields a
result with a smaller range for the parameters than the result given in \cite[Theorem.5.5]{Abeer S*}.
\item For the case $\xi=1$, Theorem \ref{application2} results coincides with the result given in \cite[Theorem.5.8]{Mahnaz C}.
\end{enumerate}
\end{remark}
\begin{theorem}
Let $0\leq \xi\leq 1$, $a>-1$, $b>-1$ and $\beta<1$ satisfies \rm(\ref{beta-Mahnaz-W-class-C}). If $f(z)\in W_{\beta}(\alpha,\gamma)$, then the function
\begin{align*}
V_\lambda(f)(z)=\int_0^1 \lambda(t) \dfrac{f(tz)}{t}dt
\end{align*}
where $\lambda(t)$ is given by
\begin{align*}
\lambda(t)= \begin{cases}
(a+1)(b+1)\frac{t^a(1-t^{b-a})}{b-a},
 &b\neq a,\\
(a+1)^2t^a\log(1/t),  &  b=a.
\end{cases}
\end{align*}
belongs to $M(\xi)$ if $a$, $b$ and $\mu$ satisfies one of the following conditions:
\begin{enumerate}
\item[{\rm{(i)}}] $b>a$, \quad $-1<a<0$ and \quad $b+a-1 <\dfrac{1}{\mu} < b-1$,
\item[{\rm{(ii)}}] $b<a$, \quad $-1<b<0$ and \quad $b+a-1 <\dfrac{1}{\mu} < a-1$,
\item[{\rm{(iii)}}] $b=a<0$  \quad and \quad $\dfrac{1}{\mu} > b-1$, which for $\mu\geq1$
(as in Theorem \ref{A.S-result-inc}) gives $b<2 \Rightarrow b<\min\{0,2\}=0$.
\end{enumerate}
\end{theorem}
\begin{proof}
To prove the required result we need to show that \eqref{initial-cond-W-class} holds.\\
For the case $a\neq b$, $\lambda(1)=0$, $\Lambda_\nu(1)=0$ and $\lambda^\prime(t)\leq 0$.
So it is enough to show the inequality \eqref{final-cond-W-class}. On substituting the values for
$\lambda(t)$, $\lambda^\prime(t)$ and $\lambda^{\prime\prime}(t)$, we need to show that $(1-\xi)A(t)+\xi B(t)\geq 0$, where
\begin{align*}
A(t)=\dfrac{(a+1)(b+1)}{(b-a)}\left(\left[\left(1+\dfrac{1}{\mu}\right)-a\right]t^a-\left[\left(1+\dfrac{1}{\mu}\right)-b\right]t^b\right)
\end{align*}
and
\begin{align*}
B(t)=\dfrac{(a+1)(b+1)}{(b-a)}\left(\left[a(a-1)-\dfrac{1}{\mu}a\right]t^a-\left[b(b-1)-\dfrac{1}{\mu}b\right]t^b\right).
\end{align*}
So it is enough to prove $A(t)>0$ and $B(t)>0$ for $0<t<1$.

Case(i) $a<b$: Since $a>-1$ and $b>-1$, so
$\dfrac{(a+1)(b+1)}{(b-a)}>0$. we need only to show that
$\displaystyle\left[\left(1+\dfrac{1}{\mu}\right)-a\right]t^a-\left[\left(1+\dfrac{1}{\mu}\right)-b\right]t^b>0$,
which clearly holds since $b>\left(1+\dfrac{1}{\mu}\right)$. Hence $A(t)>0$ for all $t\in(0,1)$. Now for
$B(t)$ to be positive, it is enough to show that
$\displaystyle\left(\left[a(a-1)-\dfrac{1}{\mu}a\right]t^a-\left[b(b-1)-\dfrac{1}{\mu}b\right]t^b\right)>0$ which
is satisfied by the given condition on $b$.

Case(ii) $b<a$: Since $a>-1$ and $b>-1$, so
$\dfrac{(a+1)(b+1)}{(b-a)}<0$. We need only to show that
$\displaystyle\left[\left(1+\dfrac{1}{\mu}\right)-a\right]t^a-\left[\left(1+\dfrac{1}{\mu}\right)-b\right]t^b<0$
which is true since $a>\left(1+\dfrac{1}{\mu}\right)$. Now for
$B(t)$ to be positive, it is enough to show that
$$
\displaystyle\left(\left[a(a-1)-\dfrac{1}{\mu}a\right]t^a-\left[b(b-1)-\dfrac{1}{\mu}b\right]t^b\right)<0
$$
which
is satisfied by the given condition on $a$.

Case(iii) $a=b\leq 0$: Changing inequality \eqref{final-cond-W-class}, which is true for $a=0$. Hence we only consider the situation
$a< 0 $. Substituting the values of $\lambda(t)$, $\lambda'(t)$ and $\lambda''(t)$ in
\eqref{final-cond-W-class}, an easy computation shows that for $0\leq \xi \leq 1$, it suffices to show that
the expressions
\begin{align*}
\log\left(\dfrac{1}{t}\right)\left(1+\dfrac{1}{\mu}-a\right)+1  \quad {\mbox{and}} \quad
\left(a(a-1)-\dfrac{a}{\mu}\right)\log\left(\dfrac{1}{t}\right)+\left(\dfrac{1}{\mu}-2a+1\right)
\end{align*}
are non-negative. Since $\displaystyle \log\left(\dfrac{1}{t}\right)$ is positive,
the non-negativity of the first expression
$ \log\left(\dfrac{1}{t}\right)\left(1+\dfrac{1}{\mu}-a\right)+1  $  follows from hypothesis (iii) of the theorem.
Similar observation shows that the second expression reduces to $a(a-1)-\dfrac{a}{\mu}\geq 0$ and
$\dfrac{a}{\mu}-2a+1 \geq 0$ using $\displaystyle \log\left(\dfrac{1}{t}\right)$ is positive. These two inequalities, for $a<0$,
gives $\dfrac{1}{\mu} \geq \max \{ 2a -1, a-1\} = a-1$, which is hypothesis (iii). The proof is complete.
\end{proof}

\begin{theorem}\label{komatu-pascu-W}
Let $c<0$ and $\mu \geq 1$ with $0\leq \xi \leq 1$. Further suppose that $p>2$ and $\beta<1$ be given by
\begin{align*}
\dfrac{\beta-1/2}{1-\beta}=&-\dfrac{(1+c)^p}{\Gamma(p)}
\int\limits_0^1t^{c}\left(\log\frac{1}{t}\right)^{p-1}q(t)dt,
\end{align*}
where $q(t)$ satisfies \eqref{de-q}. Then for
the function $f(z)\in W_{\beta}(\alpha,\gamma)$,
\begin{align*}
V_{\lambda}(f)=\frac{(c+1)^p}{\Gamma(p)}\int\limits_0^1\left(\log\frac{1}{\gamma}\right)^{p-1}t^{c-1}f(tz)dt,
\end{align*}
belongs to $M(\xi)$.
\end{theorem}

\begin{proof}
Choosing
$\displaystyle
\phi(1-t)=\left(\frac{\log(1/t)}{1-t}\right)^{p-1}$ ,
we take $C-A-B=p-1$ and $B=c+1$ so that $\lambda(t)$ takes the form
\begin{align*}
\lambda(t)=Kt^{c}(1-t)^{p-1}\phi(1-t), \quad K=\frac{(1+c)^p}{\Gamma(p)}.
\end{align*}
We complete the proof by applying Theorem \eqref{application2} and using a simple computation to obtain $c< \min\{0, 1+\dfrac{1}{\mu}\} = 0$.
\end{proof}

\section{A generalized integral operator}\label{sec-W-pascu-genl-operator}
In this section for the functions $f\in W_{\beta}(\alpha,\gamma)$, we consider
another integral operator introduced in \cite{Ali} and find the admissibility conditions to be in the class $M(\xi)$.
\begin{theorem}
Let $\mu>0$ , $\nu>0$, satisfies \eqref{eq-mu+nu}, then for $\rho<1$ and $\beta <1$ satisfying
\begin{align}\label{W-beta-rho-eq}
\dfrac{1}{2(1-\beta)(1-\rho)}=\int_0^1\lambda(t)\left[(1-\xi)\left(\dfrac{1-g(t)}{2}\right)+\xi(1-q(t))\right] dt,
\end{align}
where $g(t)$ and $q(t)$ are defined by \eqref{series-g(t)} and \eqref{series-q(t)} respectively. Assume that for $f\in\mathcal{A}$,
\begin{align*}
\mathcal{V}_\lambda(f)(z)=z\int_0^1\lambda(t)\left(\dfrac{1-\rho tz}{1-tz}\right) dt \ast f(z),
\end{align*}
then $F=\mathcal{V}_\lambda(W_{\beta}(\alpha,\gamma))\subset M(\xi) \Longleftrightarrow N_{\Pi_{\mu,\,\nu}}(h)\geq 0$.
\end{theorem}
\begin{proof}
The proof follows similar lines of proof of Theorem \ref{thm-W-Pascu-dual-equiv}. Hence, we omit details.
\end{proof}
\begin{note}
$\mathcal{V}_\lambda(f)(z)=\rho z+(1-\rho)V_\lambda(f)(z)$ and hence this operator generalize the operator
given in \eqref{eq-lambda-operator}.
\end{note}
For finding applications of the operator $\mathcal{V}_\lambda(f)(z)$, by virtue of \eqref{eqn-W-pascu-N-Pi},
Theorem \ref{app-W-class-gamma0} is ssufficient. This means we use the conditions given in
Section $\ref{sec-W-pascu-application}$. Hence we state the following results without giving their proof as they can be obtained
in a similar fashion as in the results of Section $\ref{sec-W-pascu-application}$.

\begin{corollary}\label{cor-W-pascu-genl-operator-bernardi}
Let $0\leq \xi\leq 1$, $\nu>1$, $\mu> 1$, satisfies \eqref{W-class-beta-cond},
Then for $\rho<1$ and $\beta<1$ satisfying
\begin{align*}
\dfrac{1}{2(1-\beta)(1-\rho)}=(c+1)\int_0^1 t^c\left[(1-\xi)\left(\dfrac{1-g(t)}{2}\right)+\xi(1-q(t))\right] dt
\end{align*}
where $g(t)$ and $q(t)$ are defined by \eqref{series-g(t)} and \eqref{series-q(t)}
 respectively. Assume that for $f\in W_{\beta}(\alpha,\gamma)$,
the function $\mathcal{V}_\lambda(f)(z)$ belongs to $M(\xi)$
provided
\begin{align*}
-1<c\leq min\left[\left(1+\dfrac{1}{\mu}-\dfrac{1}{\nu}\right), \left(\dfrac{1+\dfrac{1}{\mu}-\xi}{1+2\xi}\right)\right].
\end{align*}
\end{corollary}

\begin{corollary}
Let $0\leq \xi\leq 1$, $\nu>1$, $\mu> 1$, satisfies \eqref{W-class-beta-cond},
Then for $\rho<1$ and $\beta<1$ satisfying
\begin{align*}
\dfrac{1}{2(1-\beta)(1-\rho)}=k\int_0^1 t^{B-1}(1-t)^{C-A-B}\phi(1-t)\left[(1-\xi)\left(\dfrac{1-g(t)}{2}\right)+\xi(1-q(t))\right] dt
\end{align*}
where $g(t)$ and $q(t)$ are defined by \eqref{series-g(t)} and \eqref{series-q(t)}
 respectively. Assume that for $f\in W_{\beta}(\alpha,\gamma)$,
the function $\mathcal{V}_\lambda(f)(z)$ belongs to $M(\xi)$
provided
\begin{align*}
B<min\left(\left[2+\dfrac{1}{\mu}\right], (C-A-1)\right).
\end{align*}
\end{corollary}

\begin{corollary}
Let $0\leq \xi\leq 1$, $a>-1$, and $b>-1$.
Then for $\rho<1$ and $\beta<1$ satisfying
\begin{align*}
\dfrac{1}{2(1-\beta)(1-\rho)}=\int_0^1 \lambda(t)\left[(1-\xi)\left(\dfrac{1-g(t)}{2}\right)+\xi(1-q(t))\right] dt,
\end{align*}
where $g(t)$ and $q(t)$ are defined by \eqref{series-g(t)} and \eqref{series-q(t)}
 respectively and $\lambda(t)$ is given by
\begin{align*}
\lambda(t)= \begin{cases}
(a+1)(b+1)\frac{t^a(1-t^{b-a})}{b-a},
 &b\neq a,\\
(a+1)^2t^a\log(1/t),  &  b=a.
\end{cases}
\end{align*}
Assume that for $f\in W_{\beta}(\alpha,\gamma)$,
the function $\mathcal{V}_\lambda(f)(z)$ belongs to $M(\xi)$
provided
\begin{enumerate}
\item[{\rm{(i)}}] $b>a$, \quad $-1<a<0$ and \quad $b+a-1 <\dfrac{1}{\mu} < b-1$,
\item[{\rm{(ii)}}] $b<a$, \quad $-1<b<0$ and \quad $b+a-1 <\dfrac{1}{\mu} < a-1$,
\item[{\rm{(iii)}}] $b=a<0$  \quad and \quad $\dfrac{1}{\mu} > b-1$, which for $\mu\geq1$
{\rm{(as in Theorem \ref{A.S-result-inc})}} gives $b<2 \Longrightarrow b<\min\{0,2\}=0$.
\end{enumerate}
\end{corollary}

\begin{corollary}
Let $c<0$, $\mu \geq 1$ satisfies \eqref{W-class-beta-cond},
with $0\leq \xi \leq 1$
Then for $p>2$, $\rho<1$ and $\beta<1$ satisfying
\begin{align*}
\dfrac{1}{2(1-\beta)(1-\rho)}=\dfrac{(1+c)^p}{\Gamma(p)}
\int_0^1 t^{c}\left(\log\frac{1}{t}\right)^{p-1}\left[(1-\xi)\left(\dfrac{1-g(t)}{2}\right)+\xi(1-q(t))\right] dt
\end{align*}
where $g(t)$ and $q(t)$ are defined by \eqref{series-g(t)} and \eqref{series-q(t)}
respectively. Then for $f\in W_{\beta}(\alpha,\gamma)$,
the function $\mathcal{V}_\lambda(f)(z)$ belongs to $M(\xi)$.
\end{corollary}

\begin{remark}
All the above applications at $\rho=0$ reduces to the results obtained in Section $\ref{sec-W-pascu-application}$.
Further Corollary $\ref{cor-W-pascu-genl-operator-bernardi}$ at $\xi=1$  and $\xi=0$ reduces respectively
to corollaries $6.4$  and $6.5$ given in \cite{Mahnaz C}. All the other corollaries in this section for $\rho\neq 0$
and $0<\xi<1$ are not discussed in the literature elsewhere.
\end{remark}

\end{document}